\documentclass[10pt,final]{amsart}

\usepackage[english]{babel}
\usepackage{latexsym}
\usepackage{amssymb}
\usepackage{amscd}
\usepackage[latin1]{inputenc}
\usepackage[mathscr]{eucal}

\newtheorem{teor}{Theorem}[section]
\newtheorem{prop}[teor]{Proposition}
\newtheorem{cor}[teor]{Corollary}
\newtheorem{lema}[teor]{Lemma}

\theoremstyle{definition}
\newtheorem{defin}[teor]{Definition}
\newtheorem{prob}{Problem}

\theoremstyle{remark}

\newcommand{\R}{\mathbb{R}}
\newcommand{\N}{\mathbb{N}}

\newtheorem{rem}[teor]{Remark}
\newcommand{\To}{\longrightarrow}

\newcommand{\U}{\mathscr{U}}

\def\PO{\operatorname{PO}}

\def\d{\operatorname{d}}

\newcommand{\aproof}{\begin{proof}}
\newcommand{\zproof}{\end{proof}}

\def\Xii{(X_i)_{i\in I}}

\def\d{\operatorname{dist}}
\def\dens{\operatorname{dens}}
\def\dom{\operatorname{dom}}
\def\codom{\operatorname{cod}}

\def\e{\varepsilon}

\def\N{\mathbb{N}}

\def\vp{\varphi}

\def\R{\mathbb R}

\def\ran{\operatorname{ran}}

\begin{document}

\title{Banach spaces of universal disposition}

\author[A. Avil\'es, F. Cabello S\'anchez, J. M. F. Castillo, M. Gonz\'alez and Y.
Moreno]{Antonio Avil\'es, F\'elix Cabello S\'anchez, Jes\'us M. F.
Castillo, Manuel Gonz\'alez and Yolanda Moreno}

%\author{Antonio Avil\'es}
\address{Departamento de Matem\'aticas, Universidad de Murcia, 30100 Espinardo,
Murcia, Spain} \email{avileslo@um.es}

%\author{F\'elix Cabello S\'anchez}
\address{Escuela de Ingenieros T\'ecnicos Industriales, Universidad de Extremadura, Avenida de Elvas s/n, 06071 Badajoz, Spain}
             \email{fcabello@unex.es}

%\author{Jes\'us M. F. Castillo}
\address{Departamento de Matem\'aticas, Universidad de Extremadura, Avenida de Elvas s/n, 06071 Badajoz, Spain}
             \email{castillo@unex.es}

%\author{Manuel Gonz\'alez}
\address{Departamento de Matem\'aticas, Universidad de Cantabria, Avenida los Castros s/n, 39071 Santander, Spain}
             \email{manuel.gonzalez@unican.es}

%\author{Yolanda Moreno.}
\address{Escuela Polit\'ecnica, Universidad de Extremadura, Avenida de la Universidad s/n, 10071 C\'aceres, Spain}
             \email{ymoreno@unex.es}
%    Remove any unused author tags.

%    author one information

%Departamento de Matem\'aticas, Universidad de Extremadura}

\thanks{2010 Class. subject: 46A22, 46B04, 46B08,
46B26.}

\thanks{The first author was supported by MEC and FEDER (Project
MTM2008-05396), Fundaci\'{o}n S\'{e}neca (Project 08848/PI/08) and
Ram\'on y Cajal contract (RYC-2008-02051). The research of the
other four authors has been supported in part by project
MTM2010-20190}

\maketitle

\begin{abstract}In this paper we present a method to obtain Banach spaces of universal and almost-universal
disposition with respect to a given class $\mathfrak M$ of normed
spaces. The method produces, among other, the Gurari\u{\i} space
$\mathcal G$ (the only separable Banach space of almost-universal
disposition with respect to the class $\mathfrak F$ of finite
dimensional spaces), or the Kubis space $\mathcal K$ (under {\sf
CH}, the only Banach space with the density character the
continuum which is of universal disposition with respect to the
class $\mathfrak S$ of separable spaces). We moreover show that
$\mathcal K$ is not isomorphic to a subspace of any $C(K)$-space
-- which provides a partial answer to the injective space
problem-- and that --under {\sf CH}-- it is isomorphic to an
ultrapower of the Gurari\u{\i} space.

We study further properties of spaces of universal disposition:
separable injectivity, partially automorphic character and
uniqueness properties.\\

\end{abstract}

\section{Spaces of universal and almost-universal disposition}\label{sec:ud}

In   \cite{Gurariiold}  Gurari\u{\i} introduces the notions of
spaces of universal and almost-universal disposition for a given
class $\mathfrak M$ as follows.

\begin{defin}\label{unidisp} Let $\mathfrak M$ be a class of Banach spaces.
\begin{enumerate}

\item A Banach space $U$ is said to be of almost universal
disposition for the class $\mathfrak M$ if, given $A,B\in\mathfrak
M$, isometric embeddings $u: A\to U$ and $\imath:A\to B$, and
$\e>0$, there is a $(1+\e)$-isometric embedding $u': B\To U$ such
that $u=u'\imath$.

\item  A Banach space $U$ is of universal disposition for the
class $\mathfrak M$ if, given $A,B\in\mathfrak M$ and isometric
embeddings $u: A\to U$ and $\imath:A\to B$,  there is an isometric
embedding $u': B\to U$ such that $u=u'\imath$. \end{enumerate}
\end{defin}

Gurari\u{\i} shows that there exists a separable Banach space of
almost-universal disposition for the class $\mathfrak F$ of finite
dimensional spaces \cite[Theorem 2]{Gurariiold}. We recall now the
main properties of Gurari\u\i's creature. First, it is clear that
two separable Banach spaces of almost-universal disposition for
finite dimensional spaces are almost isometric
---this is shown by an obvious back-and-forth argument in
\cite[Theorem~4]{Gurariiold}. A different and simpler description
of Gurari\u\i\ space(s) by means of triangular matrices was
provided by Lazar and Lindenstrauss\ in \cite[Theorem
5.6]{lazarlind}. On the other hand, Pe\l czy\'nski and Wojtaszczyk
show in \cite{p-w} that the family of separable Lindenstrauss\
spaces has a maximal member: there is a separable Lindenstrauss\
space $\mathcal{P\!W}$ having the following property:  for every
separable Lindenstrauss\ space $X$ and each $\e>0$, there is an
operator $u:X\to  \mathcal{P\!W}$ such that $\|x\|\leq\|u(x)\|\leq
(1+\e)\|x\|$ and a contractive projection of $\mathcal{P\!W}$ onto
the range of $u$. One year later Wojtaszczyk \cite{woj} himself
shows that $\mathcal{P\!W}$ can be constructed as a space of
almost universal disposition for finite dimensional spaces.
Finally, Lusky shows in \cite{luskygura} that two separable spaces
of almost-universal disposition for finite-dimensional Banach
spaces are isometric. Therefore, there exists a unique separable
space of almost-universal disposition for finite dimensional
Banach spaces, that we will call the Gurari\u{\i} space and denote
by $\mathcal G$.

Gurari\u{\i} conjectured the existence of spaces of universal
disposition for the classes $\mathfrak F$ of finite dimensional
spaces and $\mathfrak S$ of separable spaces: see the footnote to
Theorem~5 in \cite{Gurariiold}. We will present a method able to
effectively generate such examples, as well as other spaces of
universal or almost-universal disposition, such as the
Gurari\u{\i} space \cite{Gurariiold} or the Fra\"iss\'e limit
constructed by Kubis \cite{kubis}.

\section{Background}

Our notation is fairly standard, as in \cite{lindtzaf}. A Banach
space $X$ is said to be an $\mathcal{L}_{\infty,\lambda}$-space
with $\lambda \geq 1$) if every finite dimensional subspace $F$ of
$X$ is contained in another finite dimensional subspace of $X$
whose Banach-Mazur distance to the corresponding $\ell_\infty^n$
is at most $\lambda$. A space $X$ is said to be a
$\mathcal{L}_\infty$-space if it is a
$\mathcal{L}_{\infty,\lambda}$-space for some $\lambda \geq 1$; we
will say that it is a Lindenstrauss space if it is a
$\mathcal{L}_{\infty,1 + \varepsilon}$-space for all
$\varepsilon>0$. Throughout the paper, {\sf ZFC} denotes the usual
setting of set theory with the Axiom of Choice, while {\sf CH}
denotes the continuum hypothesis ($\mathfrak c = \aleph_1$).

\subsection{The push-out construction} The push-out construction appears naturally when one
considers a couple of operators defined on the same space, in
particular in any extension problem. Let us explain why. Given
operators $\alpha:Y\to A$ and $\beta:Y\to B$, the associated
push-out diagram is
\begin{equation}\label{po-dia}
\begin{CD}
Y@>\alpha>> A\\
@V \beta VV @VV \beta' V\\
B @> \alpha' >> \PO
\end{CD}
\end{equation}
Here, the push-out space $\PO=\PO(\alpha,\beta)$ is quotient of
the direct sum $A\oplus_1 B$, the product space endowed with the
sum norm, by the closure of the subspace $\Delta=\{(\alpha
y,-\beta y): y\in Y\}$. The map $\alpha'$ is given by the
inclusion of $B$ into $A\oplus_1 B$ followed by the natural
quotient map $A\oplus_1 B\to (A\oplus_1 B)/\overline\Delta$, so
that $\alpha'(b)=(0,b)+\overline\Delta$ and, analogously,
$\beta'(a)=(a,0)+\overline\Delta$.

The diagram (\ref{po-dia}) is commutative:
$\beta'\alpha=\alpha'\beta$. Moreover, it is `minimal' in the
sense of having the following universal property: if $\beta'':A\to
C$ and $\alpha'':B\to C$ are operators such that
$\beta''\alpha=\alpha''\beta$, then there is a unique operator
$\gamma:\PO\to C$ such that $\alpha''=\gamma\alpha'$ and
$\beta''=\gamma\beta'$. Clearly, $\gamma((a, b) +
\overline{\Delta}) = \beta''(a)+\alpha''(b)$ and one has
$\|\gamma\|\leq \max \{\|\alpha''\|, \|\beta''\|\}$. Regarding the
behaviour of the maps in diagram~(\ref{po-dia}), apart from the
obvious fact that both $\alpha'$ and $\beta'$ are contractive, we
have:

\begin{lema}\label{isom}$\;$
\begin{itemize}
\item[(a)] If $\alpha$ is an isomorphic embedding, then $\Delta$
is closed. \item[(b)] If $\alpha$ is an isometric embedding and
$\|\beta\|\leq 1$ then $\alpha'$ is an isometric embedding.
\item[(c)] If $\alpha$ is an isomorphic embedding then $\alpha'$
is an isomorphic embedding. \item[(d)] If $\|\beta\|\leq 1$ and
$\alpha$ is an isomorphism then $\alpha'$ is an isomorphism and
$$\|(\alpha')^{-1}\|\leq \max \{1, \|\alpha\|\}.$$
\end{itemize}
\end{lema}
\begin{proof} (a) is clear. (b) If $\|\beta\|\leq 1$,
$$
\|\alpha'(b)\| = \|(0,b)+ \Delta\| = \inf_{y\in Y}  \|\alpha y\|+
\|b - \beta y\|  \geq  \inf_{y} \|\beta y\|+ \|b - \beta y\| \geq
\| b\|,
$$
as required. (c) is clear after (b). (d) To prove the assertion
about $(\alpha')^{-1}$, notice that for all $a\in A$ and $b\in B$
one has $(a,b) + \delta = (0, b+\beta y) + \delta$ for $y\in Y$
such that $\alpha y=a$. Therefore, for all $y'\in Y$ one has
\begin{eqnarray*} \|b +
\beta y \| &\leq & \|b + \beta y + \beta y'\| + \|\beta y'\| \\
&\leq& |b + \beta y + \beta y'\| + \|y'\| \\ &\leq& \|b + \beta y
+ \beta y'\| + \|\alpha^{-1}\| \|\alpha y'\|\end{eqnarray*} from
where the assertion follows.
\end{proof}

A Banach space $E$ is said to be (separably) injective if for
every (separable) Banach space $X$ and each subspace $Y\subset X$,
every operator $t:Y\to E$ extends to an operator $T:X\to E$. If
some extension $T$ exists with $\|T\|\leq\lambda \|t\|$ we say
that $E$ is $\lambda$-separably injective. Following
\cite{accgmsi}, a Banach space $E$ is said to be universally
separably injective if for every Banach space $X$ and each
separable subspace $Y\subset X$, every operator $t:Y\to E$ extends
to an operator $T:Y\to X$. If some extension $T$ exists with
$\|T\|\leq\lambda \|t\|$ we say that $E$ is universally
$\lambda$-separably injective.

\section{The basic construction}\label{ch:universal}

Let us consider an isometric embedding $u:A\to B$ and an operator
$t: A\to E$. We want to extend $t$ through $u$, probably at the
cost of replacing $E$ by a larger space. The push-out diagram
$$
\begin{CD}
A @> u >> B\\
@Vt VV @VV t' V\\
E@> u' >> \PO\\
\end{CD}
$$
does exactly what we ask: $t'u=u't$. It is important to realize
that $u'$ is again an isometric embedding and that $t'$ is a
contraction (resp. an isometric embedding) if $t$ is; see
Lemma~\ref{isom}. What we need is to be able to do the same with a
previously established family of embeddings. The input data for
the construction are:
\begin{itemize}
\item A Banach space $E$.
\item A family $\frak J$ of isometric embeddings between certain Banach spaces.
\item A family $\frak L$ of norm one operators from certain Banach spaces to $E$.
\end{itemize}

Each member of $\frak J$ is, by definition, an isometric embedding
$u:A\to B$, where $A$ and $B$ are Banach spaces. Then $A=\dom u$
is the domain of $u$ and $B=\codom u$ is the codomain. Let us
remark that $\codom u$ is usually larger than the image of $u$.
This implicitly yields three families of Banach spaces:
\begin{itemize}
\item $\dom \frak J=\{\dom u:u\in\frak J\}$,
\item $\codom \frak J=\{\codom u:u\in\frak J\}$,
\item $\dom \frak L=\{\dom t:t\in\frak L\}.
$\end{itemize}
To avoid complications we will assume that $\frak J$ and $\frak L$ are sets and also that $\dom \frak J=\dom \frak L$.
Notice that the only element of $\codom\frak L$ is $E$.

Set $\Gamma=\{(u,t)\in \frak J\times \frak L: \dom u=\dom t\}$ and
consider the Banach spaces of summable families $\ell_1(\Gamma,
\dom u)$ and $\ell_1(\Gamma, \codom u)$. We have an obvious
isometric embedding
$$
\oplus\frak J :\ell_1(\Gamma, \dom u)\To \ell_1(\Gamma, \codom
u)$$ defined by  $(x_{(u,t)})_{(u,t)\in\Gamma}\longmapsto
(u(x_{(u,t)}))_{(u,t)\in\Gamma} $; and a contraction
$$
\Sigma\frak L :\ell_1(\Gamma, \dom u)\To E,$$ given by
$(x_{(u,t)})_{(u,t)\in\Gamma}\longmapsto
\sum_{(u,t)\in\Gamma}t(x_{(u,t)})$. (Observe that the notation is
slightly imprecise since both $ \oplus\frak J$ and $ \Sigma\frak
L$ depend on $\Gamma$). We can form their push-out diagram
$$\begin{CD}
\ell_1(\Gamma, \dom u)@> \oplus\frak J >> \ell_1(\Gamma, \codom u)\\
@V \Sigma\frak L VV @VVV\\
E @>>> \PO.
\end{CD}
$$
In this way we obtain an isometric enlargement $\imath: E \to \PO$
such that every operator $t: A\to E$  in $\mathfrak L$ can be
extended to an operator $t':B \to \PO$ through any embedding
$u:A\to B$ in $\frak J$ provided $\dom u=\dom t= A$. In the step
$\alpha$ we leave the family $\frak J$ fixed, replace $E$ by $\PO$
and $\frak L$ by another family  $\frak L_\alpha$ and proceed
again.
%Setting $\Gamma_1=\{(u,t)\in \frak J\times \frak L_1: \dom u=\dom t\}$ we construct the push-out diagram
%$$\begin{CD}
%\ell_1(\Gamma_1, \dom u)@> \bigoplus_{\Gamma_1}\frak J >> \ell_1(\Gamma_1, \codom u)\\
%@V \sum_{\Gamma_1}\frak L VV @VVV\\
%P_1 @>>> P_2
%\end{CD}
%$$

In this way, one to iterate this construction until any countable
or uncountable ordinal. Moreover, a careful choice of the families
$\frak L_\alpha$ in the successive steps allows one to produce spaces of universal disposition, as we will see now. \\

We pass to present some specific constructions in detail. We fix a
Banach space $X$. We would take $\mathfrak J$ as the family of all
isometric embeddings between separable Banach spaces. Avoiding the
details required to fix the inconvenient that the class of
separable Banach spaces is not a set, let $\frak S$ be the family
all separable Banach spaces up to isometries. The initial step of
the construction is performed with the space $X$, the set
$\mathfrak I$ of all isometric embeddings acting between the
elements of $\frak S$ and the set $\frak L$ of all norm one
operators $t:S\to X$, where $S\in\frak S$.

We are going to define Banach spaces ${\mathfrak S}^\alpha=
{\mathfrak S}^\alpha(X)$ for all ordinals $\alpha$ starting with
$\mathfrak S^0=X$ and in such a way that, for  $\alpha <\beta$,
there is an isometric embedding
$\imath_{(\alpha,\beta)}:{\mathfrak S}^\alpha\to \mathfrak
S^\beta$. The embeddings must satisfy the obvious compatibility
condition that
$\imath_{(\beta,\gamma)}\imath_{(\alpha,\beta)}=\imath_{(\alpha,\gamma)}$
if $\alpha <\beta<\gamma$. In particular,  for each ordinal
$\alpha$, we have an embedding $\imath_{(0,\alpha)}:X\To
{\mathfrak S}^\alpha$.

We use transfinite induction as follows. Suppose $\mathfrak
S^\beta$ and the corresponding embeddings defined for each
$\beta<\alpha$. If $\alpha=\beta+1$ is a successor ordinal, we
consider the following data:
\begin{itemize}
\item The Banach space $\mathfrak S^\beta$, \item the set $\frak
I$ of all isometric embeddings acting between the elements of
$\frak S$, and \item the set $\frak L_\beta$ of all norm one
operators $t:S\to \mathfrak S^\beta$, where $S\in\frak S$.
\end{itemize}
Then we set $\Gamma_\beta=\{(u,t)\in \frak I\times \frak L_\beta: \dom u=\dom t\}$ and we form the push-out diagram
\begin{equation}\label{see0}
\begin{CD}
\ell_1(\Gamma_\beta, \dom u)@> \oplus\frak I_\beta >> \ell_1(\Gamma_\beta, \codom u)\\
@V \Sigma\frak L_\beta VV @VVV\\
\mathfrak S^\beta @>>> \PO
\end{CD}
\end{equation}
thus obtaining ${\mathfrak S}^\alpha=\mathfrak S^{\beta+1}=\PO$.
The embedding $\imath_{(\beta,\alpha)}$ is the lower arrow in the
above diagram and the other embeddings are given by composition
with $\imath_{(\beta,\alpha)}$. If $\alpha$ is a limit ordinal we
take ${\mathfrak S}^\alpha$ as the direct limit $\lim_{\beta
<\alpha} \mathfrak S^\beta$, with the obvious embeddings. Two
variations of this construction will be considered:
\begin{itemize}

\item Replace $\mathfrak S$ by the smaller family $\frak F$
generated by the finite dimensional spaces in $\frak S$. We will
call the final space $\mathfrak F^{\alpha}(X)$.

\item  We leave the initial family of Banach spaces $\frak S$, but
replace $\mathfrak I$ by the  set $\frak I^\infty$ of all
isometric embeddings of the spaces of $\frak S$ into
$\ell_\infty$, so that $\codom u=\ell_\infty$ for every $u\in\frak
I^\infty$. The choice of $\frak L$ is the same as before: all
contractive operators from the spaces in $\frak S$ to $X$. Now, we
proceed as before to construct a family of Banach spaces
$\mathfrak  U^\alpha= \mathfrak  U^\alpha(X)$ together with the
corresponding compatible linking embeddings. The passage from
$\mathfrak U^\alpha$ to $\mathfrak U^{\alpha+1}$ is as follows. We
set $\frak L_\alpha$ as the set of all norm one operators from the
spaces of $\frak S$ to $\mathfrak U^\alpha$ and
 $\Gamma_\alpha=\{(u,t)\in \frak I^\infty\times \frak L_\alpha: \dom u=\dom t\}$ and we form the push-out diagram
\begin{equation}\label{see1}
\begin{CD}
\ell_1(\Gamma_\alpha, \dom u)@> \oplus\frak I^\infty >> \ell_1(\Gamma_\alpha, \ell_\infty)\\
@V \Sigma\frak L_\alpha VV @VVV\\
\mathfrak U^\alpha @>>> \PO
\end{CD}
\end{equation}
thus obtaining $\mathfrak U^{\alpha+1}=\PO$. The embedding
$\imath_{(\alpha,\alpha+1)}$ is the lower arrow in the above
diagram and the other embeddings are given by composition with
$\imath_{(\alpha,\alpha+1)}$.
\end{itemize}

\begin{prop}\label{construction} Let $X$ be a Banach space.
\begin{itemize}
\item[(a)] The spaces $\mathfrak S^{\omega_1}(X)$ and $\mathfrak
U^{\omega_1}(X)$ are of universal disposition for separable Banach
spaces. \item[(b)] The space $\mathfrak F^{\omega_1}(X)$ is of
universal disposition for finite-dimensional Banach spaces.
\end{itemize}
\end{prop}
\begin{proof} We write the proof for $\mathfrak S^{\omega_1}=\mathfrak
S^{\omega_1}(X)$. The case $\mathfrak U^{\omega_1}(X)$ is
analogous and we leave it to the reader.

We must show that is $v:A\to B$ and $\ell: A\to \mathfrak
S^{\omega_1}$ are isometric embeddings and $B$ is separable, then
there is an isometric embedding $L:B\to \mathfrak S^{\omega_1}$
such that $Lv=\ell$. We may and do assume $A, B\in\frak S$  so
that $v$ is in $\frak I$. On the other hand there is $\alpha<
\omega_1$ such that $\ell(A)\subset \mathfrak S^\alpha$ and we may
consider that $\ell$ is one of the operators in $\frak L_\alpha$.
Therefore $\ell$ has an extension $\ell'$ making the following
square commutative:
$$
\begin{CD}
A @> v >> B\\
@V\ell VV @VV \ell' V\\
{\mathfrak S}^\alpha@>\imath_{(\alpha,\alpha+1)}>> \mathfrak
S^{\alpha+1}.
\end{CD}
$$
Actually $\ell'$ is the composition of the inclusion
$\jmath_{(v,\ell)}$ of $B=\codom v$ into the $(v,\ell)$-th
coordinate of $\ell_1(\Gamma_\alpha, \codom u)$ with the right
descending arrow in the diagram
\begin{equation}\label{see2}
\begin{CD}
\ell_1(\Gamma_\alpha, \dom u)@> \oplus\frak I >> \ell_1(\Gamma_\alpha,\codom u)\\
@V \Sigma\frak L_\alpha VV @VVV\\
{\mathfrak S}^\alpha @>>> \PO= \mathfrak S^{\alpha+1}.
\end{CD}
\end{equation}
We known that $\ell'$ is a contraction and we must prove  it is
isometric. We have
$$
\PO=({\mathfrak S}^\alpha\oplus_1 \ell_1(\Gamma_\alpha,\codom
u))/\Delta\quad\text{with}\quad\Delta=\left\{\left(\sum_{(u,t)\in\Gamma_\alpha}
tx_{(u,t)}, - \sum_{(u,t)\in\Gamma_\alpha}
ux_{(u,t)}\right)\right\}.
$$
Thus, for $b\in B$ we have $\ell'(b)=(0,\jmath_{(v,\ell)}
b)+\Delta$ and
$$
%\begin{aligned}
\|\ell'(b)\|_{\PO}=\d((0,\jmath_{(v,\ell)} b),\Delta)=\inf_{a\in
A}\left\{\|\ell(a)\|_{{\mathfrak
S}^\alpha}+\|b-v(a)\|_B\right\}=\|b\|_B
%\end{aligned}
$$
since both $\ell$ and $v$ preserve the norm.

The proof of (b) is left to the reader.
\end{proof}

This is the space Gurari\u{\i} conjectured. We will later show
that --under {\sf CH}-- such space is unique and coincides with
the Fra\"iss\'e limit in the category of separable Banach spaces
and into isometries constructed by Kubis \cite{kubis}; and also
with an ultrapower of Gurari\u{\i} space. For our purposes it is
enough to stop the constructions at $\omega_1$, however it is not
hard to believe that a careful choice of the cardinal $\alpha$ can
produce spaces $\mathfrak S^\alpha$ with special properties.
Observe that, say, $\mathfrak S^{\omega_1+1}(X)$ is not of
universal disposition for $\mathfrak S$.

\section{Properties of spaces of universal disposition}

Gurari\u{\i} shows in \cite{Gurariiold} that a space of universal
disposition for all finite-dimensional spaces cannot be separable
since no separable Banach space can be of universal disposition
for the couple $\{\R,\R\times\R\}$: indeed, a Banach space of
universal disposition for $\R$ has also been called transitive
--and the still open Mazur's rotation problem asks whether a
separable transitive Banach space must be Hilbert--; and it is
well known that the norm of a transitive space must be
differentiable at every point. This prevents the space from being
of universal disposition for $\R^2$. On the other hand it is clear
that, if the starting space $X$ is separable, then the Banach
spaces appearing in Propositions~\ref{construction} have density
character $\frak c$ since each of them is the union of an
$\omega_1$ sequence formed by Banach spaces of density $\frak c$.

\begin{prop}\label{cw1}
A Banach space of universal disposition for separable spaces must
have density at least $\frak c$ and contains an isometric copy of
each Banach space of density $\aleph_1$ or less.
\end{prop}

\begin{proof}
The first part is a juxtaposition of forthcoming
Lemma~\ref{isom-oper} --which asserts that a Banach space of
universal disposition for separable spaces must be $1$-separably
injective-- and the result in \cite{accgmsi} asserting that a
$\lambda$-separably injective space with $\lambda <2$ is either
finite-dimensional or has density character at least $\frak c$. To
prove the second part, assume that $U$ is a space of universal
disposition for separable Banach spaces and let $X$ have density
$\aleph_1$. Write $X$ as an $\omega_1$-sequence of separable
Banach spaces, beginning with $X_0=0$ and use the argument given
in the proof of (i) $\Rightarrow$ (v) in \cite[p. 221]{lindpams},
using norm preserving operators in every step.
\end{proof}

\begin{prob} Must a space of universal
disposition for finite-dimensional spaces contain an isometric
copy of each Banach space of density $\aleph_1$ or less? Our guess
is no. \end{prob}\begin{prob} Does there exist consistently a
Banach space of universal disposition for finite dimensional
spaces having density character strictly smaller than $\frak c$?
\end{prob}

\begin{lema}\label{isom-oper} Let $E$ be a Banach space. Suppose
there is a constant $\lambda$ such that for each (separable)
Banach space $X$ and every pair of into isometries $u: Y\to X$ and
$v: Y \to E$ there exists an operator $V: X \to E$ such that $Vu =
v$ with $\|{V}\|\leq \lambda$. Then $E$ is $\lambda$-(separably)
injective
\end{lema}

\begin{proof}Let $t:Y\to E$ have norm one. Denote by
$Y'$  the closure of the range of $t$ and make the push-out of
$(u, t)$:
\begin{equation}
\begin{CD}\label{looking}
 Y@> u>> X \\
 @VtVV  @VV{t'}V \\
 Y' @>u'>> \PO. \end{CD}
\end{equation}
By Lemma \ref{isom}, $u'$ is an into isometry, and the hypothesis
yields and operator $t'':\PO\to E$ such that $t''u'$ is the
inclusion of $Y'$ into $E$, with $\|t''\|\leq\lambda$.  Taking
$T=t''t'$ we end the proof.
\end{proof}

\begin{prop}\label{isom-isomo}
Let $E$ be a space of universal disposition for separable spaces.
\begin{enumerate}
\item[(a)] Given a separable Banach space $X$ and a subspace $Y
\subset X$, every isomorphic embedding $t: Y \to E$  extends to an
isomorphic embedding $T: X \to E$ with $\|T\| = \|t\|$ and  $
\|T^{-1}\| = \|t^{-1}\|$. \item[(b)] Consequently, if $\dens E
\leq \aleph_1$, then $E$ is separably automorphic; namely, any
isomorphism between two subspaces of $E$ can be extended to an
automorphism of $E$.
\end{enumerate}
\end{prop}

\begin{proof}
(a) Let $u$ denote the inclusion of $Y$ into $X$ and assume,
without loss of generality, that $\|t\|= 1$. We follow the same
notation as in Lemma \ref{isom-oper}. Looking at
Diagram~\ref{looking} we have $\|t'\|= 1$ and $u'$ is isometric,
so there is an isometric embedding $t'':\PO\to E$ such that
$t''u'$ is the inclusion of $Y'=\ran t$ into $E$. Now $T=t''t'$ is
the extension of $t$ we wanted. Clearly, $\|T\|=\|t\|= 1$.
 On the other hand, by \cite[lemma
1.3.b]{castgonz}, $\|(t')^{-1}\|\leq \max\{1,\|t^{-1}\|\}$ hence
$\|T^{-1}\|= \|t^{-1}\|$.

(b) For the second part, it suffices to show that if $Y$ is a
separable subspace of $E$, every isomorphic embedding
$\varphi_0:Y\to E$ extends to an automorphism of $E$. This is
proved through the obvious back-and-forth argument: write
$E={\bigcup_{\alpha < \omega_1} E_\alpha}$ as an
$\omega_1$-sequence of separable subspaces starting  with $E_0
=Y$. Consider the embedding $\varphi_0: E_0 \to E$. Let $\psi_1:
\varphi(E_0)+E_1\to E$ be an extension of $\varphi_0^{-1}:
\varphi(E_0)\to E$, with $\|\psi_1\|=\|\varphi_0^{-1}\|$ and
$\|\psi_1^{-1}\|=\|\varphi_0\|$. Notice that $\ran\psi_1=
E_0+\psi_1(E_1)$. Let $\varphi_2$ be the extension of
$\psi_1^{-1}$ to $E_0+\psi_1(E_1)+E_2$ provided by Part (a) and so
on. Proceeding by transfinite induction one gets a couple of
endomorphisms $\varphi$ and $\psi$ such that $\psi \varphi=
\varphi \psi ={\bf 1}_E$, with $\|\varphi\|=\|\varphi_0\|$ and
$\|\psi\|=\|\varphi_0^{-1}\|$ and $\varphi=\varphi_0$ on $Y$.
\end{proof}

Our next result shows that under {\sf CH} %\footnote{Parece que la proposicion que sigue usa la {\sf CH} solo pa que haya, en la unicidad no se necesita...},
there is no dependence on the initial separable space $X$ in the
constructions appearing in Proposition \ref{construction}

\begin{prop}\label{uuds} Under {\rm {\sf CH}} there is a unique space of universal disposition for separable spaces with density
character $\aleph_1$, up to isometries.
\end{prop}
\begin{proof} Let $X$ and $Y$ be spaces of universal disposition for separable spaces and with density
character $\aleph_1$. It is obvious that they contain isometric
copies of all separable spaces. Let us write $X= \bigcup_{\alpha <
\omega_1} X_\alpha$ and $Y= \bigcup_{\beta < \omega_1} Y_\beta$ as
increasing $\omega_1$-sequences of separable subspaces. Pick
$\beta_1$ such that there is an isometric embedding $\vp_0:X_0\to
Y_{\beta_1}$. Let $\psi_1: Y_{\beta_1}\to X$ be an isometric
extension of $\vp_0^{-1}$. As $\psi_1$ has separable range there
is $\alpha_2<\omega_1$ such that $\ran \psi_1\subset
X_{\alpha_2}$. Let $\vp_2: X_{\alpha_2}\to Y$ be an isometric
extension of $\psi_1^{-1}$. A transfinite iteration of the process
produces an isometry $X \to Y$.
\end{proof}

Let us show that there are  spaces of universal disposition for
all finite-dimensional Banach spaces which are not of universal
disposition for all separable Banach spaces.

\begin{lema}\label{finito} A $c_0$-valued operator defined on a finite-dimensional Banach space admits a compact extension
with the same norm to any superspace \end{lema}
\begin{proof} Let $F \subset X$ be a finite dimensional subspace of a Banach space $X$, and let $\tau:F\to c_0$ be a norm one operator.
Assume that $\tau = (\tau_n)$ comes defined by a pointwise null
sequence of functionals. Since $F$ is finite dimensional, the
sequence $(\tau_n)$ is actually norm null. Thus, any sequence of
Hahn-Banach extensions will also be norm null, and the operator
they define is a compact extension of $T$.
\end{proof}

\begin{prop} The space $\mathfrak F^{\omega_1}(c_0)$ is of universal disposition for finite dimensional
spaces and not of universal disposition for separable spaces.
\end{prop}
\begin{proof} It follows from Lemma \ref{finito} that the
embedding $X\to \mathfrak F^{\omega_1}(X)$ has the property that
every operator $X\to c_0$ can be extended to $\mathfrak
F^{\omega_1}(X)$. Therefore $\mathfrak F^{\omega_1}(c_0)$ contains
$c_0$ complemented, and thus it cannot be $1$-separably injective.
\end{proof}

This suggests that quite plausibly there is --even under {\sf
CH}-- a continuum of mutually non-isomorphic spaces of universal
disposition for finite-dimensional spaces. Let us show that such
is the case --outside {\sf CH}, of course-- for separable spaces.

\begin{prop} Assume that no Banach space of density character $\mathfrak c$ is
universal for all Banach spaces with density character $\mathfrak
c$. Then there is at least a continuum of non-isomorphic spaces of
universal disposition for $\mathfrak S$ with density $\mathfrak
c$.
\end{prop}
\begin{proof} We proceed by transfinite induction. To make the induction
start, form the space $\mathfrak S(1) =  \mathfrak
S^{\omega_1}(\R)$. Take, by hypothesis, a Banach space $X(1)$ with
density character $\mathfrak c$ not contained in $\mathfrak S(\R)$
and form then $\mathfrak S(2) = \mathfrak S^{\omega_1}(X(1)\oplus
\mathfrak S(1))$. Take a new Banach space $X(2)$ with density
character $\mathfrak c$ not contained $\mathfrak S(2)$ and
continue in this way.

Let $\beta<\mathfrak c$, and assume that for each $\alpha<\beta$ a
Banach space $\mathfrak S(\alpha)$ of universal disposition for
$\mathfrak S$ has already been constructed verifying:
\begin{enumerate}
\item For all $\alpha$, the space $\mathfrak S(\alpha)$ has
density character $\mathfrak c$.  \item For $\gamma \leq \alpha$
the space $\mathfrak S(\gamma)$ is isometric to a subspace of
$\mathfrak S(\alpha)$.\item For $\alpha \neq \gamma$ the spaces
$\mathfrak S(\alpha)$ and $\mathfrak S(\gamma)$ are not
isomorphic.\end{enumerate}

If $\beta=\beta'+1$ is not a limit ordinal, then get a Banach
space $X(\beta')$ with density character $\mathfrak c$ not
contained in $\mathfrak S(\beta')$ and form $\mathfrak S(\beta)=
\mathfrak S^{\omega_1}(\mathfrak S(\beta') \oplus X(\beta'))$.

If $\beta$ is a limit ordinal, then $\mathfrak S(\beta) =
\mathfrak S^{\omega_1} \left( \overline {\cup_{\alpha<\beta} \mathfrak S(\alpha)}\right)$.\\

All this yields a continuum $\mathfrak S(\alpha)), \;
\alpha<\mathfrak c$, of mutually non-isomorphic spaces of
universal disposition for separable spaces.
\end{proof}

\begin{rem} The hypothesis is consistent by a result  of Brech
and Koszmider \cite{breckosz}. The paper \cite{avibre} contains
further results on the existence of spaces of universal
disposition for $\mathfrak S$ under different cardinality
assumptions.\end{rem}

\section{Gurari\u{\i}'s space and its ultrapowers}

We can construct the Gurari\u{\i} space as follows. We fix a
countable system of isometric embeddings $\frak I_0$ having the
following density property: given an isometric embedding $w:A\to
B$ between finite dimensional spaces, and $\e>0$, there is
$u\in\frak I_0$, and surjective $(1+\e)$-isometries
$\alpha:A\to\dom u$ and $\beta: B\to\codom u$ making the square
\begin{equation}\label{wu}
\begin{CD}
A@>w>> B\\
@V\alpha VV @VV\beta V\\
\dom u@> u >> \codom u
\end{CD}
\end{equation}
commutative. Set $\frak F_0= \dom \frak I_0$.

Let now $X$ be a separable Banach space. We define an increasing
sequence of Banach spaces $G^n=G^n(X)$ as follows. We start with
$G^0=X$. Assuming $G^n$ has been defined we get $G^{n+1}$ from the
basic construction explained in Section~\ref{ch:universal} just
taking as $\mathfrak L_n$ a countable set of $G^n$-valued
contractions with domain in $\frak F_0$ such that, for every
$\e>0$, and every $(1+\e)$-isometric embedding $s:F\to G^n$, with
$F\in\frak F_0$, there is $t\in\frak L_n$ such that $\|s-t\|<\e$.
We consider the index set $\Gamma_n=\{(u,t)\in \frak I_0\times
\frak L_n: \dom u=\dom t\}$ and the push-out diagram
\begin{equation}\label{pod}
\begin{CD}
\ell_1(\Gamma_n, \dom u)@> \oplus\frak I_0 >> \ell_1(\Gamma_n, \codom u)\\
@V \Sigma\frak L_n VV @VVV\\
G^n @>>> \PO
\end{CD}
\end{equation}
Then we set $G^{n+1}=\PO$. The linking map $G^n\to G^{n+1}$ is
given by the lower arrow in the push-out diagram.

\begin{prop} Let $X$ be a separable Banach space. The space
$$G^{\omega}(X)= \lim_n G^n= \overline{\bigcup_n G^n}$$
is a separable Banach space of almost-universal disposition for
finite dimensional spaces.
\end{prop}

\begin{proof}%We prove that the dense subspace $\bigcup_n G_n$ has the required property.
Let $w:A\to B$ and $s:A\to G^\omega$ be isometric embeddings, with
$B$ a finite dimensional space and fix $\e>0$. Choose
$u\in\mathfrak I_0$, as in (\ref{wu}). Clearly, for $m$ large
enough there is a contractive $(1+\e)$-isometry $t:\dom u\to G^m$
satisfying $\|s-t\alpha\|<\e$. Let $t':\codom u\to G^{m+1}$ be the
extension provided by Diagram~\ref{pod}, so that $t$ is a
contractive $(1+\e)$-isometry such that $t'u=t$. Therefore
$t'\beta$ is a contractive $(1+\e)^2$-isometry satisfying
$\|s-t'\beta w\|\leq\e$. The following perturbation result ends
the proof.
\end{proof}

\begin{lema}
A Banach space $U$ is of almost universal disposition for
finite-dimensional spaces if and only if, given isometric
embeddings $u: A\to U$ and $\imath:A\to B$ with $B$
finite-dimensional, and $\e>0$,  there is an $(1+\e)$-isometric
embedding $u': B\to U$ such that $\|u-u'\imath\|\leq\e$.
\end{lema}

Since Gurari\u{\i} space is unique, for all separable spaces $X$,
one has $G^{\omega}(X)=\mathcal G$. Moreover, the embedding of $X$
into $G^{\omega}(X)$ enjoys  the following universal property:

\begin{prop}
Every norm one operator from $X$ into a Lindenstrauss\ space
admits, for every $\varepsilon>0$, an extension to $G^{\omega}(X)$
of norm at most $1+\varepsilon$.
\end{prop}

\begin{proof} Given $\e>0$ we fix a sequence $(\e_n)$ such that $\prod
(1+\varepsilon_n)\leq 1+\varepsilon$. Now, let $\mathcal L$ be a
Lindenstrauss\ space and $\tau: X\to \mathcal{L}$ be a norm one
operator. Look at the diagram
\begin{equation*}%\label{pod}
\begin{CD}
\ell_1(\Gamma_0, \dom u)@> \oplus\frak I_0 >> \ell_1(\Gamma_0, \codom u)\\
@V \Sigma\frak L_0 VV @VVV\\
X @>>> \PO=G^1\\
@V \tau VV\\
\mathcal L
\end{CD}
\end{equation*}
and consider the composition $\tau\circ \Sigma\frak L_0$. Since
$\mathcal L$ is a Lindenstrauss\ space, for each fixed
$(u,t)\in\Gamma_0$, the restriction of $\tau\circ \Sigma\frak L_0$
to the corresponding `coordinate' maps $\dom u=\dom t$ into a
finite-dimensional subspace of $\mathcal L$ and so it is contained
in a $(1+\varepsilon_1)$-isomorph of some finite-dimensional
$\ell_\infty^n$. Therefore, it can be extended to $\codom u$
through $u$ with norm at most $(1+\varepsilon_1)$. The
$\ell_1$-sum of all these extensions yields thus an extension $T:
\ell_1(\Gamma_0, \codom u) \to \mathcal L$ of $\tau\circ
\Sigma\frak L_0$ with norm at most $(1+\varepsilon_1)$. The
push-out property of $\PO=G^1$ yields therefore an operator
$\tau_1: G^1 \to \mathcal L$ that extends $\tau$ with norm at most
$(1+\varepsilon_1)$. Iterating the process $\omega$ times, working
with $(1+\varepsilon_n)$ at step $n$, one gets an extension
$\tau_\omega:G^\omega \to \mathcal L$ of $\tau$ with norm at most
$\prod (1+\varepsilon_n)\leq 1+\varepsilon$.
\end{proof}

Therefore every separable Banach space is isometric to a subspace
of $\mathcal G$. Taking as $X$ a separable Lindenstrauss\ space,
the universal property of the embedding yields:

\begin{cor}\label{GisPW} Every separable Lindenstrauss\ space is isometric to a
$(1+\varepsilon)$-complemented subspace of  $\mathcal G$. Hence
$\mathcal G$ is not isomorphic to a complemented subspace of any
$C(K)$-space (or, in general, any $\mathcal M$-space).
\end{cor}

Recall that an $\mathcal M$-space is a Banach lattice where
$\|x+y\|=\max\{\|x\|,\|y\|\}$ provided $x$ and $y$ are disjoint,
that is, $|x|\wedge |y|=0$. Each (abstract) $\mathcal M$-space is
representable as a (concrete) sublattice in some $C(K)$. The
second part follows from the fact, proved in \cite{bl}, that there
exist separable  Lindenstrauss\ spaces that are not complemented
subspaces of any $\mathcal M$-space. Proposition~8 in Lusky paper
\cite{lusky-survey} shows that the above Corollary is
true even with $\e=0$. \\

Everyone acquainted with ultraproducts will realize the obvious
fact that ultrapowers of Gurari\u\i\ space $\mathcal G$ are of
universal disposition for finite dimensional spaces. Less obvious
is that they also are  of universal disposition for separable
spaces. To show that, let us briefly recall the definition and
some basic properties of ultraproducts of Banach spaces. For a
detailed study of this construction at the elementary level needed
here we refer the reader to Heinrich's survey paper
\cite{heinrich} or Sims' notes \cite{sims}. Let $I$ be a set, $\U$
be an ultrafilter on $I$, and $\Xii$ a family of Banach spaces.
Then $ \ell_\infty(X_i)$ endowed with the supremum norm, is a
Banach space, and $ c_0^\U(X_i)= \{(x_i) \in \ell_\infty(X_i) :
\lim_{\U(i)} \|x_i\|=0\} $ is a closed subspace of
$\ell_\infty(X_i)$. The ultraproduct of the spaces $\Xii$
following $\U$ is defined as the quotient
$$
[X_i]_\U = {\ell_\infty(X_i)}/{c_0^\U(X_i)}.
$$
We denote by $[(x_i)]$ the element of $[X_i]_\U$ which has the
family $(x_i)$ as a representative. It is not difficult to show
that $ \|[(x_i)]\| = \lim_{\U(i)} \|x_i\|. $ In the case $X_i = X$
for all $i$, we denote the ultraproduct by $X_\U$, and call it the
ultrapower of $X$ following $\U$. If $T_i:X_i\to Y_i$ is a
uniformly bounded family of operators, the ultraproduct operator
$[T_i]_\U: [X_i]_\U\to [Y_i]_\U$ is given by $[T_i]_\U[(x_i)]=
[T_i(x_i)]$. Quite clearly, $ \|[T_i]_\U\|= \lim_{\U(i)}\|T_i\|. $

\begin{defin}\label{CI}
An ultrafilter $\U$ on a set $I$ is countably incomplete if there
is a decreasing sequence $(I_n)$ of subsets of $I$ such that
$I_n\in \U$ for all $n$, and $\bigcap_{n=1}^\infty
I_n=\varnothing$.
\end{defin}

Notice that $\U$ is countably incomplete if and only if there is a
function $n:I\to \N$ such that $n(i)\to\infty$ along $\U$
(equivalently, there is a family $\e(i)$ of strictly positive
numbers converging to zero along $\U$). It is obvious that any
countably incomplete ultrafilter is non-principal  and also that
every non-principal (or free) ultrafilter on $\N$ is countably
incomplete. Assuming all free ultrafilters  countably incomplete
is consistent with {\textsf{ZFC}}, since the cardinal of a set
supporting a free countably complete ultrafilter should
be measurable, hence strongly inaccessible.\\

We will need the following result (see \cite[II, Thm. 2.1]{hww})

\begin{teor}\label{mideal} Let $J$ be an $M$-ideal in the Banach space $E$ and $\pi:E\to E/J$
the natural quotient map. Let $Y$ be a separable Banach space and
$t:Y\to E/J$ be an operator.  Assume further that one of the
following conditions is satisfied:
\begin{enumerate}
\item $Y$ has the $\lambda$-AP. \item $J$ is a Lindenstrauss\
space.
\end{enumerate}
Then $t$ can be lifted to $E$, that is, there is an operator $T: Y
\to E$ such that $ \pi T = t$. Moreover one can get $\|T\|\leq
\lambda\|t\|$ under the assumption {\rm(1)} and $\|T\|=\|t\|$
under {\rm(2)}.
\end{teor}

One has.

\begin{prop}\label{fUs}Ultrapowers of the Gurari\u \i\ space (or more generally, of any
Banach space almost universal disposition for finite dimensional
spaces) with respect to countably incomplete ultrafilters are of
universal disposition for separable Banach spaces.
\end{prop}

\begin{proof} Let $\U$ be a countably incomplete ultrafilter on the index
set $I$. It suffices to see the following: if $S'$ is a separable
Banach space containing a subspace $S$ and we are given an (into)
isometry $u:S\to \mathcal G_{\U}$, then there is an isometry $u':
S'\to \mathcal G_{\U}$ extending $u$. We can assume and do that
$S$ has codimension 1 in $S'$. It is easy to check that $c_0^{
\mathcal U}(I,X_i)$ is an $M$-ideal in $\ell_\infty(I, M_i)$;
hence, as every Lindenstrauss\ space, $\mathcal G$ has the 1-AP,
using Theorem \ref{mideal}, the operator $u$ can be lifted to an
isometry $\tilde u: S\to \ell^\infty(I,\mathcal G)$ which we will
write as $\tilde u(x)=(u_i(x))$ for certain operators $u_i:S\to
\mathcal G$ with norm at most 1. Write
$$
S=\overline {\bigcup_{k=1}^\infty S_k},
$$where $(S_k)$ is an increasing sequence of finite dimensional
subspaces of $S$. Pick $s'\in S'\backslash S$ and let $S_k'$
denote the subspace spanned by $S_k$ and $s'$ in $S'$. Notice that
for each $x\in S$ one has $\|u_i(x)\|_\mathcal G\to \|x\|_S$
following $\U$. This implies that, given a finite dimensional
$E\subset S$ and $\e> 0$, the set
$$
\{i\in I: u_i \text{ is an $(1+\e)$-isometry on } E\}
$$
belongs to $\U$. Let $(I_n)$ be a decreasing sequence of elements
of $\mathscr U$ with intersection not in $\mathscr U$, and
consider the sets
$$
J_n=\{i\in I: u_i \text{ is an $(1+1/n)$-isometry on } S_n\} \cap
I_n.
$$
Then $J_n\in \U$ for all $n$, the sequence $(J_n)$ is decreasing
and $ {\bigcap_{n=1}^\infty J_n}=\varnothing$. For $i\in I$, we
put $n(i)=\max\{n:i\in J_n\}$. Then, of course $n(i)\to\infty$
with respect to $\U$. Next notice that, for each $i\in I$, the
operator $u_i:S_{n(i)}\to \mathcal G$ is an $(1+1/n(i))$-isometry
and so it can be extended to  an $(1+2/n(i))$-isometry
$u_i':S_{n(i)}'\to \mathcal G$. Let us consider the ultraproduct
operator
$$
w=[u_i']_{\U}: [S_{n(i)}']_{\U}\to \mathcal G_{\U}
$$
and the operator $\jmath: S'\to [S_{n(i)}']_\U$ given by $
\jmath(y)=[(x_i)], $ where $x_i$ is any point in $S_{n(i)}'$
minimizing $\d(y,S_{n(i)}')$. The map $\jmath$ is a linear
isometric embedding, and thus the sought-after extension is
$u'=w\circ\jmath$.
\end{proof}

\begin{cor} Under \textsf{CH} the following Banach spaces are all isometrically isomorphic:
\begin{itemize}
\item Any  space of universal disposition for separable spaces of
density $\aleph_1$. \item The space $\mathfrak S^{\omega_1}(\R)$.
\item The Kubis space. \item Any ultrapower of Gurari\u \i\ space
built over a nontrivial ultrafilter on the integers.
\end{itemize}
\end{cor}

\section{Remarks on the injective space problem}

One of the main open problems about injective spaces is to know if
every injective space must be isomorphic to a $C(K)$-space. The
existence of separably injective spaces which are not isomorphic
to $C(K)$-spaces has been shown in \cite{castmorestud}. We exhibit
now a $1$-universally separably injective space not isomorphic to
a $C(K)$-space. Our results are in fact much stronger: we will
actually show that spaces of universal disposition for separable
spaces cannot be isomorphic to complemented factors of any $C(K)$
or $\mathcal M$-space. Since spaces of universal disposition for
separable spaces are $1$-universally separably injective, the
previous assertion follows, in striking contrast with the facts
that that injective spaces are complemented in some $C(K)$-spaces
and $1$-injective spaces are moreover isometric to $C(K)$-spaces.

\begin{teor}\label{norMspace}
Banach spaces of universal disposition for separable spaces are
not isomorphic to complemented subspaces of $C(K)$-spaces. In
particular they are not injective.
\end{teor}

\begin{proof} Suppose $U$ is (nonzero and) of  universal disposition for
separable spaces and there is an (isomorphic) embedding $e:U\to
C(K)$ and an operator $\pi:C(K)\to U$ such that $\pi e={\bf 1}_U$.
We know that $U$ contains isometric copies of all separable Banach
spaces. Let $G_0$ be a subspace of $U$ isometric to $\mathcal G$,
$A_0$ the (closed) subalgebra spanned by $e(G_0)$ in $C(K)$ and
$B_0$ the closure of $\pi(A_0)$ in $U$. Notice that $B_0$ is a
separable subspace of $U$ containing $G_0$. As $B_0$ embeds   in
$\mathcal G$ we can find another copy $G_1$ of $\mathcal G$ inside
$U$ containing $B_0$. Now, replace $G_0$ by $G_1$ and continue
inductively. This yields a diagram (unlabelled arrows are just
inclusions)
$$
\begin{CD}
G_0@>>> B_0 @>>>  G_1@>>> B_1 @>>>  G_2@>>> \dots \\
 @V e VV @A\pi AA  @V e VV @A\pi AA  @V e VV \\
e(G_0)@>>> A_0 @>>>  e(G_1)@>>> A_1 @>>>  e(G_2)@>>> \dots
\end{CD}
$$
The space
\begin{equation}\label{see}
V=\overline{\bigcup_n G_n}= \overline{\bigcup_n B_n}.
\end{equation}
is of almost universal disposition for finite dimensional spaces,
therefore isometric to $\mathcal G$. On the other hand $e$ embeds
$V=\overline{\bigcup_n B_n}$ in $A=\overline{\bigcup_n A_{n+1}}$
while the restriction of $\pi$ to $A$ is left inverse to $e$ and
so $V$ is (isomorphic to a subspace,) complemented in $A$.
Finally, $A$ is a (separable) unital subalgebra of $C(K)$ hence it
is isometrically isomorphic to $C(M)$ for some compact
(metrizable) $M$. A contradiction with the fact that $\mathcal G$
is not a subspace of any $\mathcal M$-space (see remark after
Coro. 5.4).

\end{proof}

The proof is valid replacing $C(K)$-spaces by $\mathcal M$-spaces
(replace `algebra' by `lattice' everywhere in the proof and take
into account Benyamini's result in \cite{beny} that separable
$\mathcal M$-spaces are isomorphic to $C(K)$ spaces).

\begin{cor}\label{GUnoC}
Ultrapowers of the Gurari\u \i\ space (or more generally, of any
Banach space of almost universal disposition for finite
dimensional spaces) with respect to countably incomplete
ultrafilters are not direct factors in any $\mathcal M$-space.
\end{cor}

\begin{rem}This statement appears as Theorem~6.8 in
\cite{hensonmoore83}. Unfortunately, the argument provided by
Henson and Moore needs Stern's Lemma  \cite[Theorem
4.5(ii)]{stern}, which is wrong (see \cite{wrong}).\end{rem}

\end{document}